\documentclass{IEEEtran}
\usepackage{cite}
\usepackage{amsmath,amssymb,amsfonts}
\usepackage{algorithmic}
\usepackage{graphicx}
\usepackage{textcomp}
\def\BibTeX{{\rm B\kern-.05em{\sc i\kern-.025em b}\kern-.08em
    T\kern-.1667em\lower.7ex\hbox{E}\kern-.125emX}}       
\usepackage{color}
\usepackage{psfrag}
\usepackage{balance}
\usepackage[hidelinks]{hyperref}
\usepackage{mathrsfs}
\usepackage{graphicx}
\usepackage[shortlabels]{enumitem}
\usepackage{mathtools} 
\usepackage{amssymb,amsfonts,amsthm}
\usepackage{tikz}
\usetikzlibrary{calc}
\usepackage{color}
\usepackage{pgfplots}
\pgfplotsset{compat=newest}
\usetikzlibrary{plotmarks}
\usepgfplotslibrary{patchplots}
\usepackage{grffile}
\usepackage{amsmath}

\newtheorem{lemma}{Lemma}
\newtheorem{proposition}{Proposition}

\newtheorem{theorem}{Theorem}

\newtheorem{definition}{Definition}

\newcommand*{\QEDB}{\hfill\ensuremath{\square}}%
\usepackage{graphicx}
\usepackage{amsmath,amssymb,amsfonts,yhmath}
\usepackage{caption}
\usepackage{subcaption}
\usepackage{color}
\usetikzlibrary{arrows}
\newtheorem{prop}{Proposition}
\newtheorem{assumption}{Assumption}
\newtheorem{prope}{Property}
\newtheorem{example}{Example}

\newtheorem{rem}{Remark}

\newcommand{\nats}{\mathbb{N}}
\newcommand{\Hone}{\mathcal{H}^1((0,1);\R^{n_x})}
\newcommand{\Ltwo}{\mathcal{L}^2((0,1); \R^{n_x})}
\newcommand{\Aop}{\mathscr{R}}
\newcommand{\A}{\mathcal{A}}
\newcommand{\Xz}{\partial_z x}
\newcommand{\spec}{\operatorname{spec}}
\newcommand{\Spn}{\mathbb{S}^{n_x}_{+}}
\newcommand{\Sp}{\mathbb{S}_{+}}
\newcommand{\Dp}{\mathbb{D}^{n}_{+}}
\newcommand{\Dpn}{\mathbb{D}^{n_x}_{+}}
\newcommand{\epx}{{\epsilon}_x}
\newcommand{\epchi}{{\epsilon}_\chi}

\newcommand{\X}{\mathcal{Z}}
\newcommand{\Id}{\mathbf{I}}
\newcommand{\ep}{{\varepsilon}}
\newcommand{\Xt}{\partial_t x}
\newcommand{\dom}{\operatorname{dom}}
\newcommand{\range}{\operatorname{range}}
\newcommand{\Co}{\operatorname{co}}
\newcommand{\rk}{\operatorname{rank}}
\newcommand{\He}{\operatorname{He}}

\def\be{\begin{equation}}
  \def\eeq{\end{equation}}
\def\<{\langle}
\def\>{\rangle}

\def\bold(#1){\textbf{#1}}
\def\roma(#1){\textrm{#1}}
\def\und(#1){\underline{#1}}
\def\ove(#1){\overline{#1}}
\def\mbold(#1){\mathbb{#1}}

\def\R{\mathbb{R}}
\title{
Unknown Input Observer Design for a class of Semilinear Hyperbolic Systems with Dynamic Boundary Conditions\\ (Extended Version)}

\author{Andrea Cristofaro$^\dagger$ and Francesco Ferrante$^\ddagger$
\thanks{$^\dagger$Department of Computer, Control and Management Engineering, Sapienza University of Rome, Italy, email: andrea.cristofaro@uniroma1.it}
\thanks{$\ddagger$ Department of Engineering, University of Perugia, Via G. Duranti, 67, 06125 Perugia, Italy, email: francesco.ferrante@unipg.it}
\thanks{Authors are listed in alphabetical order.}}
\begin{document}
\maketitle
\begin{abstract}            
The problem of unknown input observer design is considered for coupled  PDE/ODE systems subject to incremental sector bounded nonlinearities and unknown boundary inputs. 
Assuming available measurements at the boundary of the distributed domain, the synthesis of the unknown input observer is based on Lyapunov methods and convex optimization. Numerical simulations support and confirm the theoretical findings, illustrating the robust estimation performances of the proposed nonlinear unknown input observer.
\end{abstract}
\section{Introduction}
\subsection{Background and motivation}
Several complex physical processes are characterized by spatial distributed phenomena and, as such, their evolution is governed by partial differential equations. Examples may be found in hydraulic networks~\cite{dos2008boundary}, multiphase flow \cite{di2011slugging}, transmission networks \cite{fliess1999active}, road traffic networks \cite{hante2009modeling} or gas flow in pipelines \cite{dick2010classical}.\\ 
Actuators and sensors are typically placed at the boundary of the spatial domain, whereas the state variables in the interior are neither directly controlled nor measured.
Classical problems such as feedback control or state estimation become particularly interesting and challenging in this context, leading to the task of designing boundary controllers and boundary observers.  Among the wide range of systems governed by PDEs, the class of hyperbolic systems has been largely considered in the control community. In this regard, sufficient conditions for controllability and
observability of first-order hyperbolic systems are discussed in \cite{li2010controllability}. The problem of stabilization of hyperbolic systems by means of boundary controls has been studied by many authors, see for instance \cite{coron2007strict} \cite{krstic2008backstepping} \cite{prieur2008robust} and the references therein. 
Observer design for linearized first-order hyperbolic systems based on Lyapunov methods has been addressed in \cite{aamo2006observer} and, for quasilinear first-order hyperbolic systems, similar results are proposed in~\cite{coron2013local} based on a backstepping approach.\\
In some cases, the evolution of a
process may have both a distributed and a finite-dimensional behavior, this resulting in a dynamics described  by partial differential equations coupled with ordinary differential equations. Typically, such interconnection takes place at the boundary of the space domain, with the output of the ODEs providing dynamic boundary conditions for the PDEs. In this context, 
boundary observer design for conservation laws with static and asymptotically stable dynamic boundary control is proposed in \cite{castillo2013boundary}, where sufficient conditions based on matrix inequalities are provided. A design procedure for backstepping observers is proposed in \cite{hasan2016boundary} based on the solution of an auxiliary set of PDEs for computing a suitable change of coordinates, while several solutions relying on Lyapunov methods have been established in \cite{trinh2017design,barreau2018lyapunov,  ferrante2019boundary}. 

\subsection{Contributions}
Adopting a setting similar to \cite{castillo2013boundary, ferrante2019boundary}, the problem of observer design in the presence of sector bounded nonlinearities~\cite{accikmecse2011observers} and unknown inputs is considered in this paper. Assuming that partial measurements are available at both boundary points of the domain, the goal is to design a full state, i.e. infinite-dimensional, nonlinear observer with the property of providing an estimation error completely decoupled from the unknown inputs. In the literature, estimators with such property are referred to as {\it Unknown Input Observers~(UIOs)} and have been recognized as an excellent tool for robust estimation and fault diagnosis \cite{chen1996design, imsland2007non, cristofaro2014fault}. In the infinite-dimensional context, the problem of linear UIO design has been previously investigated in \cite{demetriou2005unknown} from an abstract viewpoint. The related problem of active disturbance rejection is considered, for example, in \cite{feng2016new} and \cite{zhou2019performance}.

The construction proposed in this paper, which extends the preliminary work \cite{cristofaro2020unknown} by including nonlinear terms and allowing for space-dependent coefficients, relies on geometric conditions and combines the left and right boundary outputs in order to guarantee an exponentially stable error dynamics. Furthermore, a LMI-based procedure for the synthesis of observer gain matrices is given.

It is interesting to note that, unlike in the finite-dimensional case, a new feature arises in the design of Unknown Input Observers for coupled PDEs/ODEs: the finite-dimensional part of the observer is allowed to be not internally stable, whereas
a compensation is provided by the interconnection with the infinite-dimensional part which yields overall system stability.
In this paper we focus on the class of 1-D hyperbolic systems in order to deliver a constructive design procedure. However it is worth stressing that, adapting the same arguments, the constructive design of UIOs may be extended to other classes of systems governed by partial differential equations.

The remainder of the paper is organized as follows. Section~\ref{sec:SectionII} presents the problem we solve in the paper along with an outline of the proposed observer. Section~\ref{sec:Stab} provides sufficient conditions to ensure the convergence of the observer. Section~\ref{sec:Design} is devoted to the design of the observer. Finally, some numerical examples are illustrated in  Section~\ref{sec:NumExamples}. 
\subsection{Notation}
The sets $\R_{\geq 0}$ and $\R_{>0}$ represent the set of nonnegative and positive real scalars, respectively. The symbols $\Sp^n$ and $\Dp$ denote, respectively, the set of real $n\times n$ symmetric positive definite matrices and the set of diagonal positive definite matrices. For a matrix $A\in\R^{n\times m}$, $A^\top$ denotes the transpose of $A$ and when $n=m$, $\He(A)= A+A^\top$. The symbol $\otimes$ stands for the Kronecker product. For a square matrix $A\in\mathbb{R}^{n\times n}$, the spectrum is denoted by $\mathrm{spec}(A)$.
For a symmetric matrix $A$, positive definiteness (negative definiteness) and positive semidefiniteness (negative semidefiniteness) are denoted, respectively, by $A\succ 0$ ($A\prec 0$) and $A\succeq 0$ ($A\preceq 0$). Given $A, B\in\Spn$, we say that $A\preceq B$ ($A\succeq B$) if $A-B\preceq 0$  ($A-B\succeq 0$). Given $A\in\Spn$, $\lambda_{\max}(A)$ and $\lambda_{\min}(A)$ stand, respectively, for the largest and the smallest eigenvalue of $A$. 
In partitioned symmetric matrices, the symbol $\bullet$ stands for symmetric blocks. Given $x, y\in\R^n$, we denote by $\langle x, y \rangle_{\R^n}$ the standard Euclidean inner product. {Given a set $\mathbb{S}$, the notation $\mathrm{co}\mathbb{S}$ indicates its convex hull.} Let $\mathscr{U}\subset\R$ and $\mathscr{V}$ be a separable Hilbert space, we denote by $\mathcal{L}^0(\mathscr{U}; \mathscr{V})$ the set of equivalence classes of measurable functions $f\colon \mathscr{U}\rightarrow \mathscr{V}$ and, for any $p\in(0, \infty)$,  $\mathcal{L}^p(\mathscr{U}; \mathscr{V})$  stands for the set of $f\in\mathcal{L}^0(\mathscr{U}; \mathscr{V})$ such that $\Vert f\Vert_{\mathcal{L}^p}\coloneqq(\int_\mathscr{U} \vert f(x)\vert_{\mathscr{V}}
^p dx)^{\frac{1}{p}}$ is finite. Given $f, g\in\mathcal{L}^2(\mathscr{U}; \mathscr{V})$, $\langle f, g\rangle_{\mathcal{L}^2}\coloneqq \int_\mathscr{U} \langle f(x), g(x)\rangle_{\mathscr{V}} dx$. Let $\mathscr{I}$ be a bounded interval
$$
\scalebox{0.95}{$\begin{aligned}
\!\mathcal{H}^1(\mathscr{I};\mathscr{V})\!\coloneqq\!&\left\{\!f\!\in\!\mathcal{L}^2(\mathscr{I};\mathscr{V})\colon\!\!f\,\text{is absolutely continuous on}\,\,\mathscr{I}, \right.\\
&\left.\frac{d}{dz} f\in\mathcal{L}^2(\mathscr{I};\mathscr{V})\right\}
\end{aligned}$}
$$
where $\frac{d}{dz}$ stands for the weak derivative of $f$. Let $X$ and $Y$ be normed linear spaces, $f\colon U\rightarrow Y$, and $x\in U$, we denote by $Df(x)$ the Fr\'echet derivative of $f$ at $x$. 
The symbol $\mathcal{C}^k(U; V)$ denotes the set of $k$-times differentiable functions $f\colon U\rightarrow V$.
\section{Problem Statement and Outline of  Solution}
\label{sec:SectionII}

\subsection{Problem Setup}
Let $\Omega\coloneqq(0,1)$, we consider a system of $n_x$ semilinear 1-D conservation laws with dynamic boundary conditions  formally written as:
\begin{equation}
\label{eq:hyp_PDEs_3}
\def\arraystretch{1.2}
\begin{array}{lll}
&\displaystyle{\Xt(t,z)+\Lambda(z) \Xz(t, z)+Sf(Tx(t, z))=0}\\
&x(t, 0) = M\chi(t)\\
&\dot{\chi}(t)= A\chi(t)+Bu(t)+Ew(t)\\
&y_1(t)=Cx(t,0)\\
&y_2(t)=Nx(t, 1)\\
&\hspace{5cm}(t, z)\in\mathbb{R}_{\geq 0}\!\times\Omega
\end{array}
\end{equation}
where  $\Xt$ and $\Xz$ denote, respectively, the derivative of $x$ with respect to ``time'' and the ``spatial'' variable $z$, $(z\mapsto x(\cdot, z), \chi)\in\Ltwo\times\R^{n_\chi}$ is the system state, $u\in\R^{n_u}$ is a known boundary input, $w\in\R^{n_w}$ is an unknown boundary input and $y=[y_1\ y_2]\in\R^{n_y}$ is a measured output. We assume that
$\Lambda\in\mathcal{C}([0, 1];\Dpn)$, $M\in\R^{n_{x}\times n_\chi}, A\in\R^{n_\chi\times n_\chi}$, $B\in\R^{n_\chi\times n_u}$, $E\in\R^{n_\chi\times n_w}$, $C\in\R^{n_{y_1}\times n_{x}}$, $N\in\R^{n_{y_2}\times n_x}$, $S\in\R^{n_x\times n_t}$ and $T\in\R^{n_t\times n_x}$ are given and known.

Inspired by \cite{accikmecse2011observers}, henceforth, we assume that the nonlinearity $f$ satisfies the following \emph{incremental sector bound} condition:
\begin{equation}
\label{eq:incremental}
\begin{bmatrix}
y_1-y_2\\
f(y_1)-f(y_2)
\end{bmatrix}^\top
\mathfrak{M}\begin{bmatrix}
y_1-y_2\\
f(y_1)-f(y_2)
\end{bmatrix}\geq 0\quad \forall y_1, y_2\in\R^{n_t}
\end{equation}
where
$
\mathfrak{M}\coloneqq\begin{bmatrix}
-\He(U_1U_2)&U_1+U_2\\
\bullet&-\Id
\end{bmatrix}
$ and $U_2\succ U_1\succeq 0$ are given matrices. The result given next, whose proof is reported in the Appendix, shows that under the considered assumption, it turns out that $f$ is globally Lipschitz continuous. 
\begin{lemma}
\label{lemm:Lip}
Suppose that \eqref{eq:incremental} holds. Then, $f$ is globally Lipschitz continuous. \QEDB
\end{lemma}
Existence and uniqueness of mild solutions to \eqref{eq:hyp_PDEs_3} follow from  \cite[Theorem 1.2]{pazy2012semigroups} and \cite[Theorem A.6]{bastin:coron:book:2016}. Indeed the cascade PDE-ODE system~\eqref{eq:hyp_PDEs_3} can be regarded as a Lipschitz perturbation of the class of linear systems addressed in \cite{bastin:coron:book:2016} and \cite{FerranteAUT19}. The goal of the paper is to design an observer providing an exponentially convergent estimate  $(z\mapsto \hat{x}(\cdot, z),\hat{\chi})$ of the state  $(z\mapsto x(\cdot, z),\chi)$, irrespectively of the unknown input $w$.
\subsection{Outline of the Proposed Observer}
We propose an unknown input observer of the following form (the dependence on the variables $t$ and $z$ is omitted when there is no ambiguity)
\begin{equation}
\begin{array}{ll}
\partial_t \hat{x}+\Lambda(z)\partial_z\hat{x}+Sf(T\hat{x})=0&\\
\hat{x}(t,0)=M\hat{\chi}(t)&\\
\!\!\!\begin{array}{l}
\dot{\psi}=F\psi+RBu+Ky_1+L(y_2-\hat{y}_2)
\end{array}&\\
\!\!\!\begin{array}{l}
\hat{\chi}=\psi+Hy_1
\end{array}&\\
\hat{y}_2(t)=N\hat{x}(t,1)&\\
&(t, z)\in\mathbb{R}_{\geq 0}\!\times\Omega
\end{array}
\label{eq:Observer}
\end{equation}
%
where matrices $F, R, K, H, L$ are to be designed. 
At this stage, define the following two estimation errors
$\epx\coloneqq x-\hat{x}$ and $\epchi\coloneqq \chi-\hat{\chi}$. Then, the dynamics of the interconnection of observer \eqref{eq:Observer} and plant \eqref{eq:hyp_PDEs_3} in coordinates $(x, \chi, \epx, \epchi)$ can be written as follows  
\begin{subequations}
\begin{equation}
\begin{array}{ll}
&\displaystyle{\Xt+\Lambda(z) \Xz+Sf(Tx)=0}\\
&x(t, 0) = M\chi(t)\\
&\dot{\chi}= A\chi+Bu+Ew\\
&\partial_t{\epsilon}_x+\Lambda(z)\partial_z\epx+S\rho(x,\epx)=0\\
&\epx(t,0)=M\epchi(t)\\[2mm]
&\dot{\epsilon}_\chi=-F\psi+G_u u+G_w w+G_\chi \chi-LN\epsilon_x(t,1)\\[2mm]
&\quad\qquad\qquad\qquad\qquad\qquad\qquad\qquad\qquad
(t, z)\in\mathbb{R}_{\geq 0}\!\times\Omega
\end{array}
\label{eq:ErrorCoupled}
\end{equation}
where:
\begin{equation}
\label{eq:deltaf}
\rho(x,\epx)\coloneqq f(Tx)-f(T(x-\epx))\qquad \forall x,\epx\in\R^{n_x}
\end{equation}
and
\begin{equation}
\label{eq:ErrorCoupledMatrices}
\begin{aligned}
&G_u\coloneqq (I-R-HCM)B\\
&G_w\coloneqq (I-HCM)E\\
&{G_\chi\coloneqq (A-HCMA-KCM)}
\end{aligned}
\end{equation}
\end{subequations}
Paralleling the literature of unknown input observers \cite{chen1996design, cristofaro2014fault}, in the result given next we propose a selection of the observer gains $R, F$, and $K$ enabling to decouple the error dynamics \eqref{eq:ErrorCoupled} from the input $u$  and the state $\chi$.
\begin{prop}
Let $K_1\in\R^{n_\chi\times n_{y_1}}$ and select:
\begin{subequations}
\label{eq:AlgebraicConstraint}
\begin{eqnarray}
&&R=(I-HCM)\label{eq:UIO_R}\\
&&K=K_1+K_2,
\quad K_2=FH\label{eq:UIO_K}\\
&&F=A-HCMA-K_1CM\label{eq:UIO_F}
\end{eqnarray}
\end{subequations}
Then, the dynamics of the error variables in \eqref{eq:ErrorCoupled} turn into:
\begin{equation}
\label{eq:ErrorUncoupled}
\begin{array}{ll}
&\partial_t{\epsilon}_x+\Lambda(z)\partial_z{\epsilon}_x+S\rho(x,\epx)=0\\ 
&\epsilon_x(t, 0)=M\epsilon_\chi(t)\\
&\dot\epsilon_\chi=F\epsilon_\chi+REw-LN\epsilon_x(t,1)\\
[2mm]
&\quad\qquad\qquad\qquad\qquad\qquad\qquad\qquad\qquad
(t, z)\in\mathbb{R}_{\geq 0}\!\times\Omega
\end{array}
\end{equation}
\end{prop}
\begin{proof}
{Using the identity $\psi=\hat{\chi}-Hy_1=\hat{\chi}-HCM\chi$ in the last equation in \eqref{eq:ErrorCoupled}, the result can be easily proven by noticing that the selection of the matrices $R, K$, and $F$ in \eqref{eq:AlgebraicConstraint} implies that matrices $G_u$ and $G_\chi$ in \eqref{eq:ErrorCoupledMatrices} satisfy
$
G_u=0$ and $ G_\chi=F-K_2CM=F-FHCM.
$}
\end{proof}
As a second step, to decouple the error dynamics \eqref{eq:ErrorUncoupled} from the unknown input $w$, 
we select $R$, or actually $H$, such that $RE=(I-HCM)E=0$.
To ensure that a feasible solution for the above identity exists, we consider the following assumption:
\begin{assumption}\label{ass:rank}
The matrix $CME$ is full column rank.
\end{assumption}
Indeed, as long as the above assumption is in force, $CME$ is left invertible and one can pick: 
\begin{equation}\label{eq:acca}
H=E((CME)^\top CME)^{-1}(CME)^\top 
\end{equation}
This selection ensures that the matrix $G_w$ in \eqref{eq:ErrorCoupledMatrices} vanishes. Namely, the error dynamics are totally decoupled from the unknown input $w$. In particular, under \eqref{eq:AlgebraicConstraint} and \eqref{eq:acca} the error dynamics read:
\begin{equation}
\label{eq:Sys_decoup}
\begin{array}{ll}
&\partial_t{\epsilon}_x+\Lambda(z)\partial_z{\epsilon}_x+S\rho(x,\epx)=0\\ 
&\epsilon_x(t,0)=M\epsilon_\chi(t)\\
&\dot\epsilon_\chi=F\epsilon_\chi(t)-LN\epsilon_x(t,1)\\
&\qquad\qquad\qquad\qquad\qquad\qquad\qquad\qquad
(t, z)\in\mathbb{R}_{\geq 0}\!\times\Omega
\end{array}
\end{equation}
The well-posedness of error system \eqref{eq:Sys_decoup} follows from Proposition~\ref{prop:existence} in the Appendix, exploiting the same arguments used in the proof of \cite[Proposition 1]{ferrante2019boundary}.
\subsection{Abstract Formulation of the Error Dynamics}
To analyze interconnection the error dynamics \eqref{eq:Sys_decoup} we reformulate those as an abstract differential equation on the Hilbert space $\mathcal{X}\coloneqq \Ltwo\times\R^{n_\chi}$
endowed with its natural inner product, {namely for all $(x_1, \chi_1), (x_2, \chi_2)\in\mathcal{X}$ 
$$
\begin{aligned}
\langle (x_1, \chi_1), (x_2, \chi_2)\rangle_{\mathcal{X}}\coloneqq\langle x_1, x_2\rangle_{\mathcal{L}^2}+\langle \chi_1, \chi_2\rangle_{\R^{n_{\chi}}}
\end{aligned}
$$}
and consider the effect of the plant state $x$ on the error dynamics as an exogenous input. 
Consider the following linear operator $\mathscr{A}\colon\dom\mathscr{A}\rightarrow\mathcal{X}$
\begin{subequations}
\label{eq:opA}
\begin{equation}
\begin{aligned}
&\mathscr{A}\begin{pmatrix}
{\epsilon}_x\\
{\epsilon}_\chi
\end{pmatrix}\coloneqq\begin{pmatrix}
-\Lambda(z)\frac{d}{dz}&\!0\\
0&\!F
\end{pmatrix} \!\!\begin{pmatrix}
{\epsilon}_x\\
{\epsilon}_\chi
\end{pmatrix}\!+\!\begin{pmatrix}
0\\
-L N\epx(1) 
\end{pmatrix}
\end{aligned}
\label{eq:operator}
\end{equation}
with 
\begin{equation}
\begin{aligned}
\dom\mathscr{A}\coloneqq\left\{\right.&(\epx, \ep_\chi)\in\Hone\times\R^{n_\chi}\colon\\
\hspace{-2cm}&\epx(0)=M\epchi\left.\right\}
\end{aligned}
\end{equation}
\end{subequations}
and let $h\colon\Ltwo\times\Ltwo\rightarrow\mathcal{X}$ be defined as:  
\begin{equation}
(h(x,\epx))(z)\coloneqq\begin{pmatrix}
-S\rho(x(z), \epx(z))\\
0
\end{pmatrix}\qquad\forall z\in[0,1]
\label{eq:operator_B}
\end{equation}

The following result, whose proof is a direct consequence of Lemma~\ref{lemm:Lip} and \eqref{eq:incremental}, provides some useful properties for the function $h$. 
\begin{lemma}
The function $(x,\epx)\mapsto h(x,\epx)$ is globally Lipschitz continuous on $\mathcal{L}^2((0,1);\R^{n_x})$ uniformly with respect to $x$. Moreover, for all $x, \ep_x\in\mathcal{L}^2((0,1);\R^{n_x})$ and any $\pi\in\mathcal{L}^2((0,1);\R_{\geq 0})$ the following inequality holds:
\begin{equation}
\int_0^1 {\pi(z)}\begin{bmatrix}
(T\epx)(z)\\
\rho(x(z), \epx(z))
\end{bmatrix}^\top\mathfrak{M}\begin{bmatrix}
(T\epx)(z)\\
\rho(x(z), \epx(z))
\end{bmatrix}dz\geq 0
\label{eq:ConicIneq}
\end{equation}
\QEDB
\end{lemma}

{By interpreting the variable $x$ as an exogenous input given by the $x$-component of a solution to \eqref{eq:hyp_PDEs_3} with $x\in\mathcal{C}^0(\dom x, \mathcal{L}^2((0,1);\R^{n_x}))$, the error dynamics \eqref{eq:Observer} can be formally written as the following abstract differential equation on the Hilbert space $\mathcal{X}$}
\begin{equation}
\begin{pmatrix}
\dot{{\epsilon}}_x\\
\dot{{\epsilon}}_\chi
\end{pmatrix}=\mathscr{A}\begin{pmatrix}
{\epsilon}_x\\
{\epsilon}_\chi
\end{pmatrix}+h(x, \epx)
\label{eq:abstract}
\end{equation}
{Therefore, since from \cite[Section III.A]{ferrante2019boundary} $\Aop$ generates a $\mathcal{C}_0$-semigroup on $\mathcal{X}$, thanks to Lemma~\ref{lemm:Lip}, Proposition~\ref{prop:existence} entails uniqueness and global existence of maximal solutions pairs to~\eqref{eq:abstract}. \smallskip\\
Hinging on the discussion above, we aim at solving the considered state estimation problem by providing sufficient conditions for UGES of the origin of \eqref{eq:abstract}. This is the objective of the next section.}

\section{Stability analysis of the Error Dynamics}
\label{sec:Stab}
The following {\it detectability property} is instrumental to decide the approach for designing the observer gains. In particular, while stability of the observer is essentially granted under the validity of such property, in the opposite case the observer synthesis requires an additional effort.
{
\begin{prope}{[Detectability]}\label{prope:detect}
Let Assumption~\ref{ass:rank} hold, pick $H$ as defined in~(\ref{eq:acca}) and set $R=(I-HCM)$ accordingly. The pair $(RA,CM)$ is detectable.\hfill$\diamond$
\end{prope}}

The analysis of the stability of the error system depends on whether the Property \ref{prope:detect} holds or not. Therefore, we analyze the two cases separately.
\subsection{Stability with detectability}
\label{sec:StabDete}
The following result provides sufficient conditions 
for the design of the proposed observer when  
Property~\ref{prope:detect} holds.
\begin{theorem}
\label{th:StabDect}
The origin of  \eqref{eq:abstract} is UGES if there exist $\mu,\kappa\in\R_{>0}, P\in\Dpn, Q\in\mathbb{S}^{n_\chi}_{+}, K_1\in\R^{n_x\times n_{y_1}}$, and $L$ such that
\begin{align}
\label{eq:LMI_V1}
&\mathfrak{D}(z)\prec 0\quad\quad\forall z\in[0,1]\\
\label{eq:LMI_V2}
&\mathfrak{F}\coloneqq\left[
\begin{array}{cc}
-e^{-\mu} P & -N^\top L^\top Q \\
\bullet&\He(F^\top Q)+M^\top  PM\end{array}
\right]\prec 0
\end{align}
where for all $z\in [0,1]$, $\mathfrak{D}(z)$ is defined in \eqref{eq:MatrixD} (at the top of the next page) and $F=RA-K_1CM$.
\begin{figure*}
\begin{equation}
\label{eq:MatrixD}
\mathfrak{D}(z)\coloneqq\left[
\begin{array}{cc}
-\mu P-\kappa T^\top\mathrm{He}(U_1U_2)T &-\Lambda^{-1}(z)PS+\kappa T^\top( U_1+U_2)\\
\bullet&- \kappa I\end{array}
\right]
\end{equation}
\end{figure*}
\end{theorem}
\begin{proof}
{For any element of the state space $(\epx, \epchi)\in\mathcal{X}$}, let us introduce the candidate Lyapunov functional
\begin{equation}
\label{eq:FunctionalV}
V(\epsilon_x,\epsilon_\chi)\coloneqq\int_0^1e^{-\mu z}\epsilon_x^\top (z)\Lambda^{-1}(z)P\epsilon_x(z)dz+\epsilon_\chi^\top Q\epsilon_\chi
\end{equation}
In particular, observe that for all $(\epx, \epchi, x)\in\dom\mathscr{A}\times\mathcal{L}^2((0,1);\R^{n_x})$ one has
\begin{equation}
\label{eq:sandwhich}
\alpha_1\vert (\epx, \epchi)\vert^2
\leq V(\ep,\eta)\leq \alpha_2 \vert (\epx, \epchi)\vert^2
\end{equation}
where 
\begin{equation}
\begin{aligned}
&\alpha_1\coloneqq\min_{z\in[0,1]}\lambda_{\min}\left(
\begin{bmatrix}\Lambda^{-1}(z)P e^{-\mu}&0\\\bullet&Q\end{bmatrix}\right),\smallskip \\
&\alpha_2\coloneqq\max_{z\in[0,1]}\lambda_{\max}\left(\begin{bmatrix}\Lambda^{-1}(z)P&0\\\bullet&Q\end{bmatrix}\right)
\end{aligned}
\label{eq:alpha_12}
\end{equation}
are strictly positive.
Moreover, it can be easily shown that for all $(\epx, \epchi)\in\mathcal{X}$:
$$
\begin{aligned}
&(w_{x}, \hspace{-.05cm}w_{\chi})\mapsto DV(\epx, \epchi)\hspace{-.0cm}\begin{pmatrix}
w_{x}\\
w_{\chi}
\end{pmatrix}\\
&\hspace{1cm}=2\left(\int_0^1 \hspace{-.25cm} e^{-\mu z}\epx^\top (z)\Lambda^{-1}(z)Pw_{x}(z)dz+\epchi^\top Q w_{\chi}\right)
\end{aligned}
$$
Thus, {for any $(\epx, \epchi)\in\dom\mathscr{A}$, one has}
$$
\begin{aligned}
&\dot{V}(\epx, \epchi,x)\coloneqq DV(\epx, \epchi){\left(\mathscr{A}\begin{pmatrix}
{\epsilon}_x\\
{\epsilon}_\chi
\end{pmatrix}+h(x, \epx)\right)}=\\
&-2\int_0^1e^{-\mu z}\epsilon_x^\top (z)P \frac{d}{dz} \epsilon_x(z)dz\\
&\!\!\!\!\!-2\int_0^1e^{-\mu z}\epsilon_x^\top (z)\Lambda^{-1}(z)P S\rho(x(z),\epsilon_x(z))dz\\
&+\epsilon_\chi^\top \He(QF)\epsilon_\chi-2\epsilon_\chi^\top QLN\epsilon_x(1)
\end{aligned}
$$
{Integrating by parts the first term in the right-hand side and using the boundary conditions}, for $(\epx, \epchi)\in\dom\mathscr{A}$ one gets:
$$
\begin{aligned}
&\dot{V}(\epx, \epchi,x)\!=\!\int_0^1
\begin{bmatrix}
\epx(z)\\
\rho(x(z),\epx(z))\\
\epx(1)\\
\epchi
\end{bmatrix}^\top \!\!
\Psi(z)\!\begin{bmatrix}
\epx(z)\\
\rho(x(z),\epx(z))\\
\epx(1)\\
\epchi
\end{bmatrix}dz
\end{aligned}
$$
where for all $z\in[0, 1]$
$$
\!\Psi(z)\coloneqq\!\scalebox{0.75}{$\left[\!
\begin{array}{cccc}
-\mu e^{-\mu z} P &-e^{-\mu z}\Lambda^{-1}(z)PS&0&0\\
\bullet&0 &0 &0\\
\bullet&\bullet&-e^{-\mu} P & -N^\top L^\top Q\\
\bullet&\bullet&\bullet&\He(QF)+M^\top  PM
\end{array}\!
\right]$}
$$
From \eqref{eq:incremental}
$$
\begin{aligned}
\dot{V}(\epx, \epchi, x)=&\int_0^1
\begin{bmatrix}
\epx(z)\\
\rho(x(z),\epx(z))\\
\epx(1)\\
\epchi
\end{bmatrix}^\top 
\!\!\hat\Psi(z)\begin{bmatrix}
\epx(z)\\
\rho(x(z), \epx(z))\\
\epx(1)\\
\epchi
\end{bmatrix}dz\\
&\!\!\!\!\!\!\!\!\!\!\!\!\!-\kappa\int_0^1\!e^{-\mu z}\begin{bmatrix}
T\epx(z)\\
\rho(x(z), \epx(z))
\end{bmatrix}^\top\! 
\mathfrak{M}\begin{bmatrix}
T\epx(z)\\
\rho(x(z), \epx(z))
\end{bmatrix}dz
\end{aligned}
$$
with $$
\begin{array}{rl}
\hat\Psi(z)&\coloneqq\displaystyle\Psi(z)+{\kappa} e^{-\mu z}\begin{bmatrix}
\begin{bmatrix}T&0\\
0&I\end{bmatrix}^\top \mathfrak{M}\begin{bmatrix}T&0\\
0&I\end{bmatrix}&0\\
0&0
\end{bmatrix}\smallskip\\
&=e^{-\mu z}\mathrm{blkdiag}(\mathfrak{D}(z),\mathfrak{F})
\end{array}$$
where $\mathfrak{F}$ is defined in \eqref{eq:LMI_V2} and
the zero blocks have consistent dimensions.
The latter, by using \eqref{eq:LMI_V1}-\eqref{eq:LMI_V2} shows that for all $(\epx, \epchi,x)\in\dom\mathscr{A}\times\mathcal{L}^2((0,1);\R^{n_x})$
\begin{equation}
\label{eq:LyapBound}
\dot{V}(\epx, \epchi, x)\leq -\alpha_3 \vert (\epx, \epchi)\vert^2
\end{equation}
where $\alpha_3\coloneqq\displaystyle\min_{z\in[0,1]}\vert \lambda_{\max}\left(\hat\Psi(z)\right)\vert>0$. By the virtue of \eqref{eq:sandwhich} and \eqref{eq:LyapBound}, Theorem~\ref{thm:GES} ensures that the origin of \eqref{eq:abstract} is UGES. This ends the proof.
\end{proof}
\vspace{-0.4cm}
{\begin{rem}
Let us emphasize that the bottom-right block in \eqref{eq:LMI_V2} can be made negative-definite only if Property~\ref{prope:detect} is fulfilled.
\end{rem}}
\begin{rem}\label{rem:L=0}
It can be noticed that a somewhat trivial solution to the previous matrix inequality can always be found setting $L=0$. In other words, in the case of detectability of $(RA, CM)$, there is no need to use the right boundary output $y_2$ in the observer design. On the other hand, by tuning the gain $L$, the observer {transient} performance might be improved (see also \cite[Section 3]{castillo2013boundary}).
\end{rem}
\subsection{Stability without detectability}
When Property~\ref{prope:detect} does not hold, we can directly set~$K_1~=~0$, and consider the following more sophisticated Lyapunov functional inspired by the construction in \cite{ferrante2019boundary}:
\begin{equation}
W(\epsilon_x,\epsilon_\chi)\!\coloneqq\!\int_0^1\begin{bmatrix}\epsilon_x(z)\\ \epsilon_\chi \end{bmatrix}^\top \! \Pi(z)\left[
\begin{array}{c}
\epsilon_x(z)\\
\epsilon_\chi
\end{array}
\right]dz
\label{eq:LyapCross}
\end{equation}
with
$$
\Pi(z)\coloneqq\left[\!
\begin{array}{cc}
e^{-\mu z}\Lambda^{-1}(z)P&\Lambda^{-1}(z)Y^\top \\
\bullet&Q
\end{array}
\!\right]
$$
Accordingly, we can derive a sufficient  condition for exponential convergence.
\begin{theorem}
\label{thm:StabNonDet}
{Let Assumption\ref{ass:rank} hold, select $H$ as in \eqref{eq:acca} and set $F=RA$, $R=I-HCM$.} The origin of
 \eqref{eq:abstract} is UGES if there exist $\mu,\kappa\in\R_{>0}, P\in\Dpn,Y\in\R^{n_\chi\times n_x}$, $Q\in\Sp^{n_\chi}$, and $L\in\R^{n_\chi\times n_y}$ such that
\begin{align}
&\Pi(z)\succ 0\label{eq:LMI_P}&\forall z\in[0,1 ]
\smallskip\\
&\Phi(z)\coloneqq\begin{bmatrix}
\Upsilon(z)&\Gamma(z)\\
\bullet&-\kappa I
\end{bmatrix}\prec0\label{eq:LMI_Mz}&\forall z\in[0,1 ]
\end{align}
where for any $z\in[0,1]$, $\Upsilon(z)$ and $\Gamma(z)$ are defined as in
\eqref{eq:Mz} (at the top of the next page). 
\end{theorem}
\begin{proof} 
\begin{figure*}
\begin{equation}
\label{eq:Mz}
\begin{aligned}
&\Upsilon(z)\coloneqq\begin{bmatrix} 
-\mu e^{-\mu z} P -\kappa T^\top\mathrm{He}(U_1U_2)T& -\Lambda^{-1}(z)Y^\top  L N & \Lambda^{-1}(z)Y^\top  F\\
\bullet&-e^{-\mu z} P &-Y^\top -N^\top  L^\top Q\\
\bullet&\bullet& \He(QF+Y M)+M^\top  PM\\
\end{bmatrix}\\
&\Gamma(z)\coloneqq \begin{bmatrix} 
\kappa T^\top (U_1+ U_2)-e^{-\mu z}\Lambda^{-1}(z)PS\\
0\\
-Y\Lambda^{-1}(z) S\\
\end{bmatrix}
\end{aligned}
\end{equation}
\end{figure*}
\begin{figure*}
\begin{equation}
\label{eq:MatrixQ}
\mathfrak{Q}(D)\coloneqq\left[
\begin{array}{cc}
-\mu P-\kappa T^\top\mathrm{He}(U_1U_2)T &-D PS+\kappa T^\top( U_1+U_2)\\
\bullet&- \kappa I\end{array}
\right]
\end{equation}
\end{figure*}
The proof follows the same lines as the proof of Theorem~\ref{th:StabDect}. In particular, let, for all $(\epx, \epchi)\in\mathcal{X}$, $W$ be defined as in \eqref{eq:LyapCross}. Then, for all $(\epx, \epchi)\in\mathcal{X}$ one has:
\begin{equation}
\label{eq:sandwhich_W}
\beta_1\vert (\epx, \epchi)\vert^2_\mathcal{X}
\leq W(\epx,\epchi)\leq \beta_2 \vert (\epx, \epchi)\vert^2_\mathcal{X}
\end{equation}
where 
\begin{equation}
\beta_1\coloneqq\min_{z\in[0,1]}\lambda_{\min}\left(\Pi(z)\right),\quad
\beta_2\coloneqq\max_{z\in[0,1]}\lambda_{\max}\left(\Pi(z)\right)
\label{eq:alpha_12_W}
\end{equation}
are strictly positive due to \eqref{eq:LMI_P}.
Similarly as in \cite{ferrante2019boundary} and exploiting \eqref{eq:ConicIneq}, for all $(\epx, \epchi,x)\in\dom\mathscr{A}\times\mathcal{L}^2((0,1);\R^{n_x})$ we have
\begin{equation}
\begin{aligned}
\dot{W}(\epx, \epchi,x)&
\coloneqq {DW(\epx, \epchi)\left(\mathscr{A}\begin{pmatrix}
{\epsilon}_x\\
{\epsilon}_\chi
\end{pmatrix}+h(x, \epx)\right)}=\\
&\int_0^1
\varpi(z)^\top 
\Phi(z)\varpi(z) dz\\
&-\kappa\int_0^1\begin{bmatrix}
T\epx(z)\\
\rho(x(z), \epx(z))
\end{bmatrix}^\top
\mathfrak{M}\begin{bmatrix}
T\epx(z)\\
\rho(x(z), \epx(z))
\end{bmatrix}dz
\end{aligned}
\label{eq:Vdot2}
\end{equation}
where $\varpi(z)\coloneqq(
\epx(z),
\epx(1),
\epchi,
\rho(x(z), \epx(z)))$. At this stage notice that from \eqref{eq:LMI_Mz},  there exists $\beta_3>0$ such that: 
\begin{equation}
\label{eq:BoundPhi}
\Phi(z)\preceq-\beta_3 I\qquad \forall z\in[0, 1]
\end{equation} and
in particular $\beta_3\coloneqq\min_{z\in[0,1]}\vert \lambda_{\max}\left(\hat\Phi(z)\right)\vert$. Combining \eqref{eq:Vdot2} and \eqref{eq:BoundPhi} gives for all $(\epx, \epchi,x)\in\dom\mathscr{A}\times\mathcal{L}^2((0,1);\R^{n_x})$:
\begin{equation}
\dot{W}(\epx, \epchi,x)\leq -\beta_3 \vert (\epx, \epchi)\vert^2
\label{eq:LyapDotTwo}
\end{equation}
Thanks to \eqref{eq:sandwhich_W} and \eqref{eq:LyapDotTwo}, Theorem~\ref{thm:GES} ensures that the origin of \eqref{eq:abstract} is UGES. This concludes the proof.
\end{proof}
\section{Observer Design}
\label{sec:Design}
In this section, we show how the proposed observer can be designed via the solution to some linear matrix inequalities coupled to a line search on two scalar variables. 
\subsection{Observer Design with Detectability}
We assume that Property~\ref{prope:detect} holds. In this scenario, the design of the observer can be performed by relying on Theorem~\ref{th:StabDect}. On the other hand, the conditions in Theorem~\ref{th:StabDect} have two drawbacks that render their use tricky. First, they are nonlinear in the decision variables. Second, they need to be satisfied {over an uncountable set, {\it i.e.}, the interval $[0,1]$.} To overcome these drawbacks, next we propose a set of computationally affordable sufficient conditions for the design of the observer. 
\begin{proposition}
\label{prop:designDete} 
Assume that Property~\ref{prope:detect} is fulfilled. For $i=1, 2,\dots, n_x$, and all $z\in [0,1]$, let $d_i(z)$ be the $(i,i)$ entry of the matrix $\Lambda^{-1}(z)$. Define for all $i=1,2,\dots, n_x$, $\displaystyle\overline{d}_i\coloneqq\min_{z\in[0, 1]}d_i(z)$ and $\displaystyle\overline{d}_i\coloneqq\max_{z\in[0, 1]}d_i(z)$ and let 
$$
\mathbb{G}\coloneqq\left\{D\in\Dpn\colon D=\sum_{i=1}^{n_x}r_i(\mathbf{e}_i\otimes\mathbf{e}_i), r_i\in\{\underline{d}_i, \overline{d}_i\}\right\}
$$
where $\{\mathbf{e}_i\}_{i=1,..,n_x}$ stands for the canonical basis. 
The origin of \eqref{eq:abstract} is UGES if there exist $\mu,\kappa\in\R_{>0}, P\in\Dpn, Q\in\mathsf{S}^{n_\chi}_{+}, X\in\R^{n_\chi\times n_{y_1}}$ and $J\in\R^{n_\chi\times n_{y_2}}$ such that:
\begin{align}
\label{eq:LMI_V1_design}
&\mathfrak{Q}(D)\prec 0,\qquad\forall D\in\mathbb{G}\\
\label{eq:LMI_V2_design}
&\left[
\begin{array}{cc}
-e^{-\mu} P & -N^\top J^\top\\
\bullet&\He(QRA-XCM)+M^\top  PM\end{array}
\right]\prec 0
\end{align}
where $\mathfrak{Q}(D)$ is defined in \eqref{eq:MatrixQ} (at the top of the page). In particular, the observer gains can be selected as $L=Q^{-1}J, K_1=Q^{-1}X$.
\end{proposition}
\begin{proof}
As first step, notice that by replacing the expressions of $L$ and $K_1$ into \eqref{eq:LMI_V2_design} one gets \eqref{eq:LMI_V2}. To show that \eqref{eq:LMI_V1_design} implies \eqref{eq:LMI_V1},
first observe that by continuity of the functions $z\mapsto d_i(z)$, it turns out that 
$$
\begin{aligned}
&\range(z\mapsto \Lambda^{-1}(z))\subset\\
&\qquad\{D\in\Dpn\colon 
D=\sum_{i=1}^{n_x}r_i(\mathbf{e}_i\otimes\mathbf{e}_i), r_i\in[\underline{d}_i, \overline{d}_i]\}=\Co\mathbb{G}
\end{aligned}
$$
where the last equality is straightforward. Now, observe that since $D\mapsto\mathfrak{Q}(D)$ is affine, $\mathfrak{Q}(\Co\mathbb{G})=\Co\mathfrak{Q}(\mathbb{G})$. 
Hence, \eqref{eq:LMI_V1_design} implies that, for all $z\in [0, 1]$, $\mathfrak{Q}(\Lambda(z))\prec 0$. The latter, by simply  observing that for all $z\in [0, 1]$ $\mathfrak{Q}(\Lambda(z))=\mathfrak{D}(z)$, shows that  \eqref{eq:LMI_V1_design} implies \eqref{eq:LMI_V1}.
\end{proof}
\subsection{Observer Design without Detectability}
Following similar arguments as in the previous subsection, the following result can be established to deal with the scenario outlined in Section~\ref{sec:StabDete}. 
\begin{figure*}
\begin{equation}
\label{eq:MatrixXi}
\begin{aligned}
&\Xi(\omega, D)\coloneqq\begin{bmatrix} 
-\mu \omega P -\kappa T^\top\mathrm{He}(U_1U_2)T& 0 & DY^\top  F&\kappa T^\top (U_1+ U_2)-\omega DPS\\
\bullet&-\omega P &-Y^\top -N^\top  J^\top&0 \\
\bullet&\bullet& \He(QF+Y M)+M^\top  PM&-YDS\\
\bullet&\bullet&\bullet&-\kappa I
\end{bmatrix}
\end{aligned}
\end{equation}
\end{figure*}
\begin{proposition}
\label{prop:design_non_detect}
For $i=1, 2,\dots, n_x$, and all $z\in [0,1]$, let $s_i(z)$ be the $(i,i)$ entry of the matrix $e^{\mu z}\Lambda^{-1}(z)$. Define for all $i=1,2,\dots, n_x$, $\displaystyle\overline{s}_i\coloneqq\min_{z\in[0, 1]}s_i(z)$ and $\displaystyle\overline{s}_i\coloneqq\max_{z\in[0, 1]}s_i(z)$ and let 
$$
\widehat{\mathbb{G}}\coloneqq\left\{\widehat{D}\in\Dpn\colon \widehat{D}=\sum_{i=1}^{n_x}r_i(\mathbf{e}_i\otimes\mathbf{e}_i), r_i\in\{\underline{s}_i, \overline{s}_i\}\right\}
$$
The origin of \eqref{eq:abstract} is UGES if there exist $\mu,\kappa,\theta\in\R_{>0}, P\in\Dpn, Q\in\mathbb{S}^{n_\chi}_{+}$ and $J\in\R^{n_\chi\times n_{y_2}}$ such that:
\begin{align}
\label{eq:LMI2_V2_design_PiPos}
&\begin{bmatrix}
\widehat{D} P&\widehat{D} Y^\top\\
\bullet&Q
\end{bmatrix}\succ 0&\forall \widehat{D}\in\widehat{\mathbb{G}}\\
\label{eq:LMI2_V2_design}
&\left[
\begin{array}{ccc}
\Xi(1, D)&\Theta(D)&\theta Z^\top J^\top\\
\bullet&-\theta Q&0\\
\bullet&\bullet&-\theta Q
\end{array}
\right]\prec 0&\forall D\in\mathbb{G}\\
\label{eq:LMI2_V2_design2}
&\left[
\begin{array}{ccc}
\Xi(e^{-\mu}, D)&\Theta(D)&\theta Z^\top J^\top\\
\bullet&-\theta Q&0\\
\bullet&\bullet&-\theta Q
\end{array}
\right]\prec 0& \forall D\in\mathbb{G}
\end{align}
where $\Xi(\omega, D)$ is defined in \eqref{eq:MatrixXi} (at the top of the next page) and 
$$
\Theta(D)\coloneqq\begin{bmatrix}
-Y D&0&0&0
\end{bmatrix}^\top,\quad Z\coloneqq\begin{bmatrix}
0&N^\top&0&0
\end{bmatrix}.$$
In particular, the observer gain can be selected as $L=Q^{-1}J$.
\end{proposition}
\begin{proof}
As a first step, notice that since for all $z\in[0, 1]$, $\Lambda^{-1}(z)e^{\mu z}\in\Co\widehat{\mathbb{G}}$, by following similar arguments as in the proof of Proposition~\ref{prop:designDete}, it follows that the satisfaction of \eqref{eq:LMI2_V2_design_PiPos} assures
$$
\begin{bmatrix}
e^{\mu z}\Lambda^{-1}(z) P&e^{\mu z}\Lambda^{-1}(z) Y^\top\\
\bullet&Q
\end{bmatrix}\succ 0\qquad \forall z\in[0, 1]
$$
which, by a simple congruence transformation, is equivalent to \eqref{eq:LMI_P}.
To conclude the proof, let
$$
\mathfrak{P}(\omega, D)\coloneqq\begin{bmatrix}
\mathfrak{U}(\omega, D)&\mathfrak{G}(\omega, D)\\
\bullet&-\kappa I
\end{bmatrix}
$$
where $\mathfrak{U}(\omega, D)$ and $\mathfrak{G}(\omega, D)$ are defined, respectively, by replacing in $\Upsilon(z)$ and $\Gamma(z)$, $\Lambda^{-1}(z)$ and $e^{-\mu z}$ by $D$ and $\omega$, which gives for all $z\in[0,1]$,
$\mathfrak{P}(e^{-\mu z},\Lambda^{-1}(z))=\Phi(z)$.
Moreover, as pointed out in the proof of Proposition~\ref{prop:designDete},
$\range(z\mapsto\Lambda^{-1}(z))\subset\Co\mathbb{G}$ and similarly
$\range(z\mapsto e^{-\mu z})=[e^{-\mu}, 1]$.
Hence, using the fact that $\mathfrak{P}$ is biaffine with respect to $\omega$ and $D$, it turns out that \eqref{eq:LMI_Mz} holds if
\begin{equation}
\label{eq:VdotNonDete}
\begin{aligned}
&\mathfrak{P}(1, D)\prec 0, &\mathfrak{P}(e^{-\mu}, D)\prec 0,\qquad\forall D\in\mathbb{G}.
\end{aligned}
\end{equation}
The proof can be completed by noting that, using the bounding technique in \cite[Proposition 4]{FerranteAUT19},  \eqref{eq:LMI2_V2_design} and \eqref{eq:LMI2_V2_design2} imply \eqref{eq:VdotNonDete}. 
\end{proof}

\section{Numerical Examples}
\label{sec:NumExamples}
In this section, we showcase the application of the proposed observer in two numerical examples. 
\begin{example}
\label{ex:Exa1}
In this first example, we consider the following data:
$$
A=\left[\begin{array}{ccc} 0 & 1 & 0\\ -1 & 0 & 0\\ 0 & 0 & 0 \end{array}\right],\ B=0,\ E=\left[\begin{array}{c} 0\\ 0\\ 1 \end{array}\right], N=\begin{bmatrix} 1 &0\\ 0&1\end{bmatrix}$$ 
$M=\left[\begin{array}{ccc} 1 & 1 & 1\\ 0 & 1 & 0 \end{array}\right],\ \Lambda(z)=\left[\begin{array}{ccc} 1+z^2&0\\ 0&e^{-z}\end{array}\right], S=\begin{bmatrix}
1\\
0
\end{bmatrix}$, 
$T=\begin{bmatrix}
1&0
\end{bmatrix}, C=\begin{bmatrix}1&1\end{bmatrix}, f(y)=\tanh(y)$.
In particular, for the above selection of $f$, it can be easily shown that $U_1=0$ and $U_2=\frac{1}{2}$; see, e.g., \cite{accikmecse2011observers}.
From \eqref{eq:acca} and \eqref{eq:AlgebraicConstraint} one gets
$$
H^\top=\left[\begin{smallmatrix} 0& 0.333& 0.333 \end{smallmatrix}\right], R=\left[\begin{smallmatrix}1 & 0 & 0\\ -0.333 & 0.333 & -0.333\\ -0.333 & -0.667 & 0.667 \end{smallmatrix}\right]$$ 
At this stage, notice that $\spec(RA)\approx\{0,-0.167\pm j 0.553\}$ and that 
$
\rk\left[\begin{smallmatrix}
RA\\
CM
\end{smallmatrix}\right]=3$,
thereby showing that the pair $(RA, CM)$ is detectable, i.e., Property~\ref{prope:detect} holds. Thus, the design of the observer can be performed via the use of Proposition~\ref{prop:designDete}. Following this approach, we obtain:
$$
\begin{aligned}
&\mu=0.1, \kappa=13.75, P=\left[\begin{smallmatrix}
9.28 & 0\\ 0 & 14.76\end{smallmatrix}\right], L= 0\\
&Q=\left[\begin{smallmatrix} 21.96 & -6.954 & 4.087\\ -6.954 & 16.13 & -1.041\\ 4.087 & -1.041 & 10.13\end{smallmatrix}\right], K_1^\top=\left[\begin{smallmatrix} 0.6597 & 0.4537 & 0.358\end{smallmatrix}\right].
\end{aligned}
$$
Indeed, {as mentioned in Remark~\ref{rem:L=0}}, in this case there is no need of using the output $y_2$. To finalize the design of the observer, we plug the gains $R$ and $K_1$  into \eqref{eq:UIO_F} and \eqref{eq:UIO_K} to get: 
$$
\begin{aligned}
&F=\left[\begin{smallmatrix} -0.6597 & -0.3194 & -0.6597\\ -0.7871 & -1.241 & -0.4537\\ 0.3086 & -1.049 & -0.358
\end{smallmatrix}\right]\\
&K_2^\top=\left[\begin{smallmatrix} -0.3264 & -0.5648 & -0.4692\end{smallmatrix}\right].
\end{aligned}
$$ To validate our theoretical findings, next we present some simulations of the proposed observer\footnote{Numerical integration of hyperbolic PDEs is performed via the use of the Lax-Friedrichs (Shampine's two-step variant) scheme implemented in Matlab\textsuperscript{\tiny\textregistered} by Shampine \cite{shampine2005solving}.}. In these simulations $w(t)=\sin(2t)$ and 
\begin{equation}
\begin{array}{lll}
&x_1(0, z)=0.5(\sin(2\pi z)-1)&\forall z\in[0, 1]\\
&x_2(0, z)=0.5(\sin(4\pi z)-1)&\forall z\in[0, 1]\\
&\chi(0)=(0.5, -0.5, -0.5)\\
&\hat{x}(0, z)=0&\forall z\in[0, 1]\\
&\hat{\chi}(0)=0
\end{array}
\label{eq:init_cond_1}
\end{equation}
The effectiveness of the proposed observer in reconstructing the plant state is confirmed by \figurename~\ref{fig:error_example_1_norm} in which the evolution of norm of the estimation error is reported. 
\end{example}
\begin{figure}[t!]
\centering
\psfrag{t}[][][1]{$t$}
\psfrag{y}[][][0.9]{$\vert (\epx,\epchi) \vert^2$}
\includegraphics[width=0.8\columnwidth, trim={0cm 0.5cm 0cm 0cm},clip]{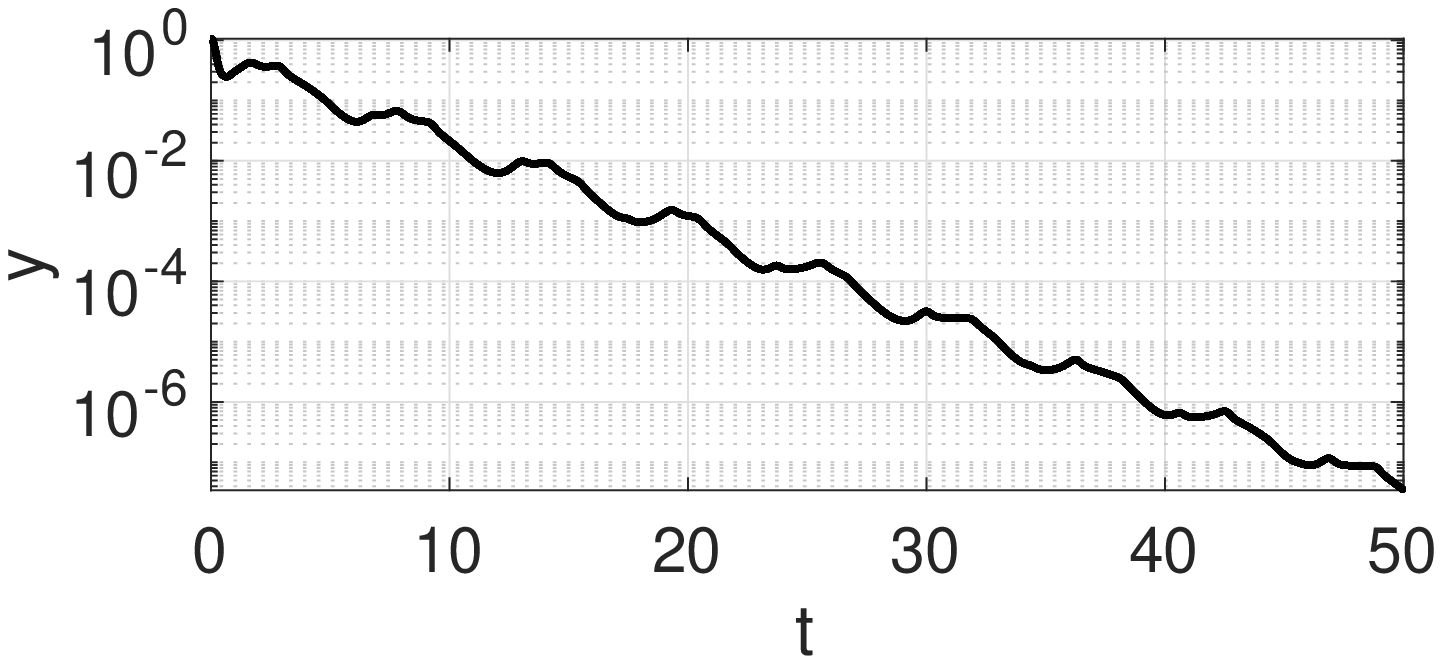}
\caption{Evolution of the squared norm of the estimation error ($y$-axis is in log scale) from the initial condition in \eqref{eq:init_cond_1}.}
\label{fig:error_example_1_norm}
\end{figure}
\begin{example}
In this second scenario, we select $A, B$,$E$,$N$,$S$, $M$, and $f$ as in Example~\ref{ex:Exa1} and
$\Lambda=\begin{bmatrix}\sqrt{2}&0\\  0&2\end{bmatrix}, T=\begin{bmatrix}0&1\end{bmatrix},
C= \begin{bmatrix}1   &  0\end{bmatrix}
$.
From \eqref{eq:acca} one gets $H^\top=\left[\begin{array}{ccc} 0& 0& 1 \end{array}\right]$.
Hence using \eqref{eq:AlgebraicConstraint} with $K_1=0$ yields:
$$
R=\left[\begin{smallmatrix} 1 & 0 & 0\\ 0 & 1 & 0\\ -1 & -1 & 0 \end{smallmatrix}\right], F=\left[\begin{smallmatrix} 0 & 1 & 0\\ -1 & 0 & 0\\ 1 & -1 & 0 \end{smallmatrix}\right], K_2=0$$
Notice that the pair $(RA, CM)$ is not detectable. This can be checked by using the PBH-test of observability: in particular, it easy to verify that
$
\rk\left[\begin{smallmatrix}
RA\pm j I\\
CM
\end{smallmatrix}\right]=2$.
This prevents one from using Proposition~\ref{prop:designDete}  to design an unknown input observer. 
To overcome this problem, we make use of Proposition~\ref{prop:design_non_detect}. Following this second approach, we obtain
$$
\begin{array}{c}
L=\left[\begin{smallmatrix}
-0.01306 & -0.02304\\ 0.008322 & 0.279\\ 0.1568 & -0.3331
\end{smallmatrix}\right], \mu=1, \theta=1$,  $P=\left[\begin{smallmatrix} 9.28 & 0\\ 0 & 14.76\end{smallmatrix}\right],\medskip\\ Y=\left[\begin{smallmatrix}-1.146 & 0.2738\\ -2.177 & -11.58\\ -1.684 & 0.5115\end{smallmatrix}\right]$,  $Q=\left[\begin{smallmatrix} 21.96 & -6.954 & 4.087\\ -6.954 & 16.13 & -1.041\\ 4.087 & -1.041 & 10.13\end{smallmatrix}\right]
\end{array}$$ In these simulations initial conditions and the unknown input $w$ are unchanged with respect to the previous example.
\end{example}
\begin{figure}[t!]
\centering
\psfrag{t}[][][1]{$t$}
\psfrag{y}[][][0.9]{$\vert (\epx,\epchi) \vert^2$}
\includegraphics[width=0.8\columnwidth, trim={0cm 0.5cm 0cm 0cm},clip]{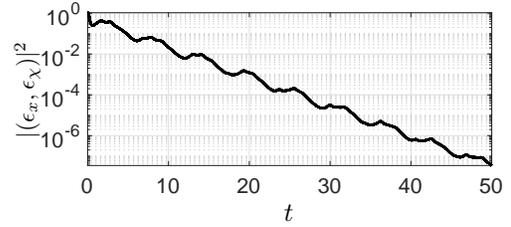}
\caption{Evolution of the squared norm of estimation error ($y$-axis is in log scale) from the initial condition in \eqref{eq:init_cond_1}.}
\label{fig:error_ODE_example_2_norm}
\end{figure}
\section{Conclusion}
The design of unknown input observers (UIO) for a class of coupled PDE/ODE semilinear systems subject to unknown boundary inputs has been addressed in this paper. In particular, distributed nonlinearities are assumed to satisfy an incremental sector boundedness condition.
The structure of the proposed UIOs is analogous to their finite-dimensional version, and the interconnection with the system to be estimated is made by means of injection of the boundary outputs. The synthesis of observer parameters and gains is based on Lyapunov methods and convex optimization techniques. An interesting feature of the considered infinite-dimensional UIO is that, even though the finite-dimensional part lacks of detectability in some cases {(namely, the dynamic system driving the boundary conditions might be not detectable)}, the complete error system may still be made exponentially stable {by using the injection of the right boundary output.} Properties and performance of the observer are illustrated through numerical simulations.
{Future research directions include the use of the proposed estimation scheme to generate fault detection/isolation filters \cite{chen1996design, cristofaro2014fault} and the application of LMI-based techniques to the synthesis of unknown input observers for other classes of distributed parameter systems.}
\bibliographystyle{IEEEtran}
\bibliography{bibTAC_UIO}
\appendix[Auxiliary results and definitions]
We consider abstract semilinear dynamical systems of the form:
\begin{equation}
\label{eq:AbstractODE}
\dot{\zeta}=\mathscr{R}\zeta+\sigma(\upsilon, \zeta)
\end{equation}
where $\zeta\in\mathcal{Z}$ is the system state, $\upsilon\in\mathcal{U}$ is the system input, $\X$ is the state space that we assume to be a separable real Hilbert space endowed with the inner product $\langle\cdot,\cdot\rangle_\X$, $\mathcal{U}$ is the input space, 
$\mathscr{R}\colon\dom\mathscr{A}\subset\X\rightarrow\X$ is a linear operator, and $\sigma\colon\X\times\mathcal{U}\rightarrow\X$ satisfies the following assumptions: 
\begin{itemize}
\item[$(i)$] for all $\zeta\in\X$, $\upsilon\mapsto \sigma(\upsilon, \zeta)$ is continuous; 
\item[$(ii)$] $\sigma$ is globally Lipschitz continuous uniformly in $\upsilon$, i.e., there exists $\ell\geq0$ 
such that for all $\upsilon\in\mathcal{U}$, $\zeta_1, \zeta_2\in\X$
$$
\!\!\!\!\langle \sigma(\upsilon, \zeta_1)- \sigma(\upsilon, \zeta_2),\sigma(\upsilon, \zeta_1)- \sigma(\upsilon, \zeta_2)\rangle\leq\ell\langle \zeta_1-\zeta_2, \zeta_1-\zeta_2\rangle
$$
\end{itemize}
\begin{definition}
\label{def:sol}
A pair $(\varphi, \upsilon)\in\mathcal{C}^0(\dom\varphi; \X)\times\mathcal{L}^0(\dom \upsilon; \mathcal{U})$ is a solution pair to \eqref{eq:AbstractODE} if $\dom\varphi=\dom \upsilon$, $\dom\varphi$ is an interval of $\R_{\geq 0}$ including zero, $\upsilon$ is locally essentially bounded, and for all $t\in\dom\varphi$
$$
\begin{aligned}
&\int_0^t \varphi(s)ds\in\dom\Aop\\
&\varphi(t)=\varphi(0)+ \Aop \int_0^t \varphi(s)ds+
\int_0^t \sigma(\upsilon(s), \varphi(s))ds
\end{aligned}
$$ 
Moreover, we say that $(\varphi, u)$ is maximal if its domain cannot be extended and it is global if $\sup\dom\varphi=\infty$.
\hfill$\circ$
\end{definition}
\medskip

The result given next establishes existence and uniqueness of maximal solution pairs to~\eqref{eq:AbstractODE}.
\begin{proposition}
\label{prop:existence}
Let $\Aop$ be the infinitesimal generator of a $\mathcal{C}_0$-semigroup on $\X$ and let $\sigma$ be globally Lipschitz uniformly in $u$. Then, for any $\xi\in\X$ and $u\in\mathcal{L}^0(\dom u; \mathcal{U})$ locally essentially bounded, there exists a unique maximal solution pair $(\varphi, u)$ to \eqref{eq:AbstractODE} such that $\varphi(0)=\xi$. Moreover, such a pair satisfies:
\begin{equation}
\label{eq:SemigroupRep}
\varphi(t)=\mathscr{T}(t)\xi+\int_{0}^t\mathscr{T}(t-s)\sigma(u(s),\varphi(s))ds\qquad\forall t\in\dom \varphi
\end{equation}
where $\mathscr{T}(t)$ is the $\mathcal{C}_0$-semigroup generated by $\Aop$ on $\mathcal{Z}$. In addition, if $\xi\in\dom\Aop$, $(\varphi, u)$ is a strong solution pair to \eqref{eq:AbstractODE}, namely $\varphi$ is differentiable almost everywhere, $\dot{\varphi}$ is locally integrable, and
\begin{equation}
\label{eq:almost_sol}
\dot{\varphi}(t)=\Aop\varphi(t)+\sigma(u(t), \varphi(t))\qquad\text{for\,\,almost\,\,all}\,\,t\in\dom\varphi
\end{equation}
\end{proposition}
\begin{proof}
As a first step, we show that under the considered assumptions,  ``integrated solutions'' as in Definition~\ref{def:sol} and solutions to the integral equation \eqref{eq:SemigroupRep} coincide. From \cite[page 218]{magal2018theory}, since $\dom\Aop$ is dense in $\mathcal{X}$ ($\Aop$ is the infinitesimal generator of a strongly continuous semigroup on $\mathcal{X}$) it follows that  any ``integrated solution'' as in Definition~\ref{def:sol} satisfies 
\eqref{eq:SemigroupRep}. Now we show the other implication. 
Let $(\varphi, u)$ be a solution pair to \eqref{eq:SemigroupRep} with $\xi\in\mathcal{X}$. Notice that since $\Aop$ is the infinitesimal generator of a strongly continuous semigroup on $\mathcal{X}$ and $f$ is globally Lipschitz continuous (uniformly in $u$), a simple adaptation of \cite[Theorem 1.2., page 185]{pazy2012semigroups} enables to show that for all $\xi\in\mathcal{X}$ and all $u\in\mathcal{L}^0(\dom u;\mathcal{U})$, there exists a unique $\varphi$ with $\dom\varphi=\dom u$ satisfying \eqref{eq:SemigroupRep}. 
For the sake of exposition, let us rewrite $\varphi$ as follows:
\begin{equation}
\label{eq:solSemi}
\varphi(t)=\mathscr{T}(t)\xi+\Psi(t)\qquad\forall t\in\dom\varphi
\end{equation}
where:
$$
\Psi(t)\coloneqq\int_{0}^t \mathscr{T}(t-\theta)f(u(\theta), \varphi(\theta))d\theta
$$
Observe that $\varphi$ is continuous. Moreover, since $\Aop$ is the infinitesimal generator of the strongly continuous semigroup $\{\mathscr{T}(t)\}_{t\in\R_{\geq 0}}$, from 
\cite[Theorem 2.1.10, page 21]{CurtainZwart:95} it follows that for all $t\in\R_{\geq 0}$
\begin{equation}
\label{eq:SemDom}
\begin{aligned}
&\int_{0}^t\mathscr{T}(s)\xi ds\in\dom\Aop,&\Aop\int_{0}^t\mathscr{T}(s)\xi ds=\mathscr{T}(t)\xi-\xi
\end{aligned}
\end{equation}
Pick any $t\in\dom\varphi$, then one has:
$$
\int_0^t\Psi(s)ds=\int_{0}^t \int_{0}^s \mathscr{T}(s-\theta)f(u(\theta), \varphi(\theta))d\theta ds
$$
We now make use of Fubini's theorem to obtain a simpler expression of the above righthand side integral term.  To this end, observe that $(s,\theta)\mapsto \mathscr{T}(s-\theta)f(u(\theta), \varphi(\theta))$ is Lebesgue measurable with respect to the product $\sigma$-algebra generated by rectangular sets of the form $X\times Y$, with $X, Y\subset\R$ Lebesgue measurable. This can be proven by noticing that $\theta\mapsto \mathscr{T}(s-\theta)f(u(\theta), \varphi(\theta))$ is Lebesgue measurable\footnote{Lebesgue measurability of 
$\theta\mapsto \mathscr{T}(s-\theta)f(u(\theta), \varphi(\theta))$ can be proven by using the fact that $\mathscr{T}$ is strongly continuous and $\theta\mapsto f(u(\theta), \varphi(\theta))$ is locally Lebesgue integrable; see \cite[Example A.5.14, page 625]{CurtainZwart:95}. In fact, local Lebesgue integrability of $\theta\mapsto f(u(\theta), \varphi(\theta))$ follows from $u$ being measurable, $u\mapsto f(u, x)$ being continuous for any $x\in\mathcal{X}$, and $f$ being globally Lipschitz continuous uniformly in $u$.} and that $s\mapsto \mathscr{T}(s-\theta)f(u(\theta), \varphi(\theta))$ is continuous. Moreover, by relying on the fact that $\mathscr{T}$ is strongly continuous and that $\theta\mapsto f(u(\theta), \varphi(\theta))$ is locally Lebesgue integrable, so is $(s,\theta)\mapsto \mathscr{T}(s-\theta)f(u(\theta), \varphi(\theta))$. 
Hence, a direct application of Fubini's theorem, along with a simple change of variables, yields
\begin{equation}
\label{eq:Fubini}
\int_0^t\Psi(s)ds=\int_{0}^t \underbrace{\left(\int_{0}^{t-\theta} \mathscr{T}(\rho)f(u(\theta), \varphi(\theta))d\rho\right)}_{g_t(\theta)} d\theta
\end{equation}
At this stage notice that since $\Aop$ is the infinitesimal generator of the strongly continuous semigroup $\{\mathscr{T}(t)\}_{t\in\R_{\geq 0}}$, from 
\cite[Theorem 2.1.10, page 21]{CurtainZwart:95} it follows that for all $\theta\in[0, t]$, $g_t(\theta)\in\dom\Aop$ and in particular 
\begin{equation}
\Aop g_t(\theta)=\mathscr{T}(t-\theta)f(u(\theta), \varphi(\theta))-f(u(\theta), \varphi(\theta))\quad\forall \theta\in[0, t]
\label{eq:IntegralAg}
\end{equation}
Bearing in mind \eqref{eq:Fubini}, Fubini's theorem enables to conclude that for all $\theta\in[0, t]$, $g_t\in\mathcal{L}_1((0, \theta); \mathcal{X})$ and, thanks to \eqref{eq:IntegralAg} and the assumptions considered on $f$, $u$, and $\mathscr{T}$ being strongly continuous, it turns out that $\Aop g_t\in\mathcal{L}_1((0, \theta); \mathcal{X})$. Thus, since $\Aop$ is closed ($\Aop$ is the infinitesimal generator of a strongly continuous semigroup; see \cite[Theorem 2.1.10, page 21]{CurtainZwart:95}) by applying \cite[Theorem 3.7.12]{hille1996functional} and using \eqref{eq:Fubini}, it follows that for all $t\in\dom\varphi$
\begin{equation}
\label{eq:SemDomBelong}
\int_0^t\Psi(s)ds\overset{\eqref{eq:Fubini}}{=}\int_0^t g_t(\theta)d\theta\in\dom\Aop
\end{equation}
and in particular
\begin{equation}
\label{eq:SemDom2}
\begin{aligned}
&\Aop\int_0^t\Psi(s)ds=\int_0^t\Aop g_t(\theta)d\theta
\end{aligned}
\end{equation}
Therefore, from \eqref{eq:SemDom} and \eqref{eq:SemDomBelong} it follows that for all $t\in\dom\varphi$
$$
\int_0^t\varphi(s)ds\in\dom\Aop
$$
Now we show that $\varphi$ satisfies the identify in Definition~\ref{def:sol}. Let $t\in\dom\varphi$, from \eqref{eq:SemDom2}, by using \eqref{eq:IntegralAg} one gets
$$
\begin{aligned}
\Aop\int_0^t\Psi(s)ds\!&=\\
&\int_0^t (\mathscr{T}(t-\theta)f(u(\theta), \varphi(\theta))-f(u(\theta),\varphi(\theta)))d\theta
\end{aligned}
$$
Hence, by recalling \eqref{eq:solSemi} and combining the above expression with \eqref{eq:SemDom} one has:
$$
\begin{aligned}
\Aop\int_0^t\!\!\!\varphi(s)ds=&\mathscr{T}(t)\xi-\xi\\
&\!\!\!\!\!\!\!\!\!\!\!\!\!\!\!\!+\int_0^t (\mathscr{T}(t-\theta)f(u(\theta),\varphi(\theta))-f(u(\theta), \varphi(\theta)))d\theta
\end{aligned}
$$
Therefore using the expression of $\varphi$ in \eqref{eq:SemigroupRep}, one has
$$
\varphi(t)-\Aop\int_0^t \varphi(s)ds=\xi+\int_0^t f(u(\theta), \varphi(\theta))d\theta
$$
which, recalling that $\varphi(0)=\xi$, corresponds to the identity in Definition~\ref{def:sol}. The steps carried out so far enable to fully replace the notion of solution in Definition~\ref{def:sol} with \eqref{eq:SemigroupRep}. In particular, due to the above mentioned uniqueness of maximal solutions  to \eqref{eq:SemigroupRep}, existence and uniqueness of maximal solutions to \eqref{eq:AbstractODE} is established.

The proof can be simply concluded by showing that if $\xi\in\dom\Aop$, the unique solution to \eqref{eq:SemigroupRep} is differentiable almost everywhere, 
it satisfies \eqref{eq:almost_sol}, and $\dot{\varphi}$ is locally integrable due to 
$$
\varphi(t)=\varphi(0)+\int_{0}^t (\Aop\varphi(s)+f(u(s), \varphi(s)))ds\quad\forall t\in\dom\varphi
$$
namely $\varphi$ is locally absolutely continuous.   
\end{proof}
\begin{definition}
We say that the origin of \eqref{eq:AbstractODE} is uniformly (with respect to $u$) globally exponentially stable (\emph{UGES}) if any maximal solution pair $(\varphi, u)$ to \eqref{eq:AbstractODE} is global and it satisfies:
$$
\vert \varphi(t)\vert\leq \kappa e^{-\lambda t}\vert \varphi(0)\vert,\qquad\forall t\in\dom\varphi
$$ 
for some (solution independent) $\lambda, \kappa>0$.
\end{definition}
The following result provides sufficient condition for UGES. 
\begin{theorem}
\label{thm:GES}
Let $\Aop$ be the infinitesimal generator of a $\mathcal{C}_0$-semigroup on $\mathcal{Z}$. Assume that there exists a Fréchet differentiable functional $V\colon\X\rightarrow\R$ and positive scalars $\alpha_1, \alpha_2, \alpha_3$, and $p$ such that the following items hold:
\begin{itemize}
\item[($i$)] For all $\zeta\in\X$
$$
\alpha_1\vert \zeta\vert^p\leq V(\zeta)\leq\alpha_2\vert \zeta\vert^p
$$ 
\item[($ii$)] For all $\zeta\in\dom\Aop$, $\upsilon\in\mathcal{U}$
$$
DV(\zeta)(\Aop \zeta+\sigma(\upsilon, \zeta))\leq-\alpha_3\vert \zeta\vert^p.
$$ 
\end{itemize}
Then, the origin of  \eqref{eq:AbstractODE} is UGES for \eqref{eq:AbstractODE}.\QEDB
\end{theorem}
\begin{proof}
First we prove UGES of the origin for strong solutions to \eqref{eq:AbstractODE}; see Proposition~\ref{prop:existence}. In particular, pick $\xi\in\dom\Aop$ and $u\in\mathcal{L}^0(\dom u;\mathcal{U})$, and let $(\varphi, u)$ be the unique maximal solution to \eqref{eq:AbstractODE} 
with $\varphi(0)=\xi$. Then, from Proposition~\ref{prop:existence} it follows that $\varphi$ is differentiable almost everywhere and $\dot{\varphi}(t)=\Aop\varphi(t)+f(u(t), \varphi(t))$.
In particular, notice that since strong solutions are solutions in the sense of Definition~\ref{def:sol}, it turns out that for all $t\in\dom\varphi$:
\begin{equation}
\label{eq:AbsCont}
\begin{aligned}
\varphi(t)=&\varphi(0)+\Aop\int_0^t\varphi(s)ds+\int_0^tf(u(s), \varphi(s))ds\\
&=\varphi(0)+\int_0^t\Aop\varphi(s)ds+\int_0^t f(u(s),\varphi(s))ds\\
&=\varphi(0)+\int_0^t \dot{\varphi}(s)ds
\end{aligned}
\end{equation}
where the first identify follows from \cite[Theorem 3.7.12]{hille1996functional} due to $\Aop$ being closed, $\varphi(t)\in\dom\Aop$ for almost all $t\in\dom\varphi$, and $\Aop\varphi$ being locally  integrable.
Namely, \eqref{eq:AbsCont} shows that $\varphi$ is locally absolutely continuous on $\dom\varphi$.
Define for all $t\geq 0$, $W(t)\coloneqq V(\varphi(t))$.
Observe that since $V$ is Fréchet differentiable everywhere and $\varphi$ is differentiable almost everywhere, it follows that for almost all $t\in\dom\varphi$
$$
\dot{W}(t)\coloneqq \frac{d}{dt}W(\varphi(t))=DV(\varphi(t))\dot{\varphi}(t) 
$$
In particular, since $V$ is Fréchet differentiable (and so locally Lipschitz), and $\dot{\varphi}$ is locally integrable, if follows that $W=V\circ\varphi$ is locally absolutely continuous on $\dom\varphi$. Then, since $\varphi$ is a strong solution to \eqref{eq:AbstractODE}, for all almost all $t\dom\varphi$ one has:
$$
\dot{W}(t)=DV(\varphi(t))(\Aop\varphi(t)+f(u(t), \varphi(t)))
$$
Hence, using items $(i)$ and $(ii)$, for almost all $t\in\dom\varphi$ one has:
$$
\dot{W}(t)\leq-\frac{\alpha_3}{\alpha_2}W(t)
$$
which, by using the fact that $W$ is locally absolutely continuous on $\dom\varphi$, from the comparison lemma \cite[Lemma 3.4, page 102]{Khalil} for all $t\in\dom\varphi$, one gets:
$$
W(t)\leq W(0)e^{-\frac{\alpha_3}{\alpha_2}t}\qquad\forall t\dom\varphi
$$
the latter, thanks to item $(i)$, gives:
$$
\vert \varphi(t)\vert_\A\leq \left(\frac{\alpha_2}{\alpha_1}\right)^{\frac{1}{p}} e^{-\frac{\alpha_3}{p\alpha_2} t}\vert \varphi(0)\vert_\A\qquad\forall t\in\dom\varphi
$$
thereby showing uniform global exponential stability of the origin for strong solutions. We now extend the proof to all solutions to \eqref{eq:AbstractODE}. Let $(\varphi, u)$ be a maximal solution to \eqref{eq:AbstractODE}. Since by assumption $\Aop$ is the infinitesimal generator of a $\mathcal{C}_0$ semigroup $\{\mathcal{T}(t)\}_{t\in\R_{\geq 0}}$ on $\X$, from \cite{arendt2011vector}  it follows that $\dom\Aop$ is dense in $\X$. Thus, there exists a sequence $\{\xi_k\}\subset\dom\Aop$ such that $\xi_k\rightarrow\varphi(0)$. Now, let for all $k\in\nats$
$$
\varphi_k(t)=\mathscr{T}(t)\xi_k+\int_{0}^t\mathscr{T}(t-s)f(u(s), \varphi_k(s))\qquad\forall t\dom\varphi
$$
Observe that for all $k\in\nats$, thanks to Proposition~\ref{prop:existence}, $\varphi_k$ is a strong solution to \eqref{eq:AbstractODE}. Thus for all $k$, as shown earlier:
\begin{equation}
\vert\varphi_k(t)\vert\leq \left(\frac{\alpha_2}{\alpha_1}\right)^{\frac{1}{p}} e^{-\frac{\alpha_3}{p\alpha_2} t}\vert \varphi_k(0)\vert\qquad\forall t\in\dom\varphi
\label{eq:BoundGESk}
\end{equation}
Pick any $t\in\dom\varphi$. Then, for all $k\in\mathbb{N}$
$$
\begin{aligned}
&\vert\varphi_k(t)-\varphi(t)\vert=\left\vert\mathscr{T}(t)(\xi_k-\xi)\right.\\
&\qquad\left.+\int_0^t \mathscr{T}(t-s)(f(u(s), \varphi_k(s))-f(u(s), \varphi(s)))ds\right\vert\\
&\qquad\leq M \left\vert \xi_k-\xi\right\vert+M \ell \int_0^t \left\vert \varphi_k(s)-\varphi(s)\right\vert ds
\end{aligned}
$$
where $M\coloneqq\displaystyle\max_{s\in [0, t]}\Vert \mathscr{T}(t)\Vert$. Hence the latter, via Grönwall's inequality, yields
$$
\vert\varphi_k(t)-\varphi(t)\vert\leq M \left\vert \xi_k-\xi\right\vert e^{M\ell t}
$$
This shows that, for all $t\in\dom\varphi$
\begin{equation}
\lim_{k\rightarrow \infty}\vert\varphi_k(t)-\varphi(t)\vert=0
\label{eq:NormConv}
\end{equation}
the latter, combined with \eqref{eq:BoundGESk}, via a simple application of the triangle inequality, enable to show that \eqref{eq:BoundGESk} holds for $\varphi$. This concludes the proof.
\end{proof}
\begin{proof}[Proof of Lemma \ref{lemm:Lip}]
From \eqref{eq:incremental}, one has that for all $y_1, y_2\in\R^{n_t}$
$$
-\delta_f^\top \delta_f-\delta_y^\top \He(S_1S_2)\delta_y+2\delta_y^\top(S_1+S_2)\delta_f\geq 0
$$
where $\delta_f\coloneqq f(y_2)-f(y_1)$ and $\delta_y\coloneqq y_2-y_1$. The latter, thanks to Young's inequality, yields for all $\digamma>0$
$$
\frac{1-\digamma}{\digamma}\delta_f^\top\delta_f+\delta_y^\top\left(-\He(S_1S_2)+\digamma\|S_1+S_2\|\Id
\right)\delta_y\geq 0
$$
Pick $\digamma>\max\{1, \lambda_{\max}(\He(S_1S_2))/\|S_1+S_2\|\}$, then the above inequality yields:
$$
g_1\delta_f^\top\delta_f +g_2\delta_y^\top \delta_y\geq 0
$$
with $g_1\coloneqq\frac{1-\digamma}{\digamma}<0$ and $g_2\coloneqq\lambda_{\max}\left(-\He(S_1S_2)+\digamma\Id
\right)>0
$, which in turn shows that $f$ is globally Lipschitz continuous with Lipschitz constant $\ell=g_2\vert g_1^{-1}\vert$. 
\end{proof}
\end{document}